\documentclass[12pt]{amsart}   \usepackage{amssymb,latexsym}
\usepackage{amssymb,latexsym}
\usepackage{amsfonts,mathrsfs}
\usepackage[margin=1.4in]{geometry}
\usepackage{fancyhdr}

\newtheorem{theorem}{Theorem}[section]
\newtheorem{lemma}[theorem]{Lemma}

\theoremstyle{remark}

\begin{document}
\title{ THE ASYMPTOTIC BEHAVIOR OF QUSI-HARMONIC FUNCTIONS AND EIGENFUNCTIONS OF DRIFT LAPLACIAN AT INFINITY}
\thanks{\noindent \textbf{MR(2010)Subject Classification}   47F05 58C40}
\author{min chen}
\author{jiayu Li}
\author{yuchen Bi}
\address[Corresponding author] {University of Science and Technology of China, No.96, JinZhai Road Baohe District,Hefei,Anhui, 230026,P.R.China.}

\email{cmcm@mail.ustc.edu.cn}

\thanks{The research is supported by the National Nature Science Foudation of China  No. 11721101 No. 11526212 }

\pagestyle{fancy}
\fancyhf{}
\renewcommand{\headrulewidth}{0pt}
\fancyhead[CE]{}
\fancyhead[CO]{\leftmark}
\fancyhead[LE,RO]{\thepage}

\begin{abstract}
Note that $\mathbb{R}^m$ with the metric $g=e^{-\frac{|x|^2}{2(m-2)}}ds_0^2$ is actually a  Riemannian manifold with a singularity at $\infty.$ The metric  is quite singular at infinity and it is not complete. Colding-Minicozzi \cite{11} pointed out that the Ricci curvature of this metric does not have a sign and goes to negative infinity at infinity and thus there is no way to smoothly extend the metric to a neighborhood of infinity. Chen-Li \cite{7}  proved that any non-constant quasi-harmonic function or eigenfunction of drift Laplacian is discontinuous at infinity. In this paper, we show expansions of quasi-harmonic functions and of  eigenfunctions of drift-Laplacian  in terms of spherical harmonics. Using these expansions, we have a more precise description of the asymptotic behavior of quasi-harmonic functions and of eigenfunctions of drift-Laplacian at infinity. Moreover, we improve the Liouville theorem of quasi-harmonic functions and eigenfunctions of drift-Laplacian by reducing the requirement of the conditions.
\end{abstract}

\maketitle

\numberwithin{equation}{section}
\section{Introduction}
\setcounter{equation}{0}
Chen-Li \cite{7} studied eigenfunctions of Quasi-Laplacian $\Delta_g = e^{\frac{|x|^2}{2(m-2)}}(\Delta_{g_0} - \nabla_{g_0} h \cdot \nabla_{g_0}) = e^{\frac{|x|^2}{2(m-2)}} \Delta_h $ for $h = \frac{|x|^2}{4}$ and proved that any non-constant quasi-harmonic function is discontinuous at infinity in Corollary 3.3 and any non-constant eigenfunction of drift Laplacian $\Delta_h = \Delta_{g_0} - \nabla_{g_0} h \cdot \nabla_{g_0}$ is discontinuous at infinity in Theorem 1.4, which means that any quasi-harmonic function or eigenfunction of drift Laplacian could not  converge to a constant at infinity. 

Recently, Colding-Minicozzi defined a frequency function $U(r) = \frac{r}{2} (\log{I})'$ where $I(r) = r^{1-n} \int_{\partial B_r} u^2$ and used the mean value $\sqrt{I}$ of $u$ to measure the rate of the growth of eigenfunctions of drift Laplacian $\mathcal{L} = \Delta_{h}$  in Theorem 4.8 in \cite{5} and Theorem 1.1 in \cite{6}. 
\begin{theorem}\label {thm:1.4}
\cite{6} Given $\epsilon > 0$ and $\delta > 0$, there exist $r_1 > 0$ such that if $\mathcal{L}u = -\lambda u$ and $U(\bar{r}_1) \ge \delta + 2 \sup\{0, \lambda\}$ for some $\bar{r}_1 \ge r_1$, then for all $r \ge R(\bar{r}_1)$
\begin{align*}
	U(r) > \frac{r^2}{2} - n - 2\lambda - \epsilon.
\end{align*}
\end{theorem}

$U(r) > \frac{r^2}{2} - n - 2\lambda - \epsilon$ implies that  $\sqrt{I} \ge Ce^{\frac{r^2}{4}} r^{-(n + 2\lambda - \epsilon)}$. Theorem 1.4 shows that there is a sharp dichotomy for the growth of eigenfunctions of $\mathcal{L}$: either $\sqrt{I} \le Cr^{2(\delta + 2\lambda)}$, $u$ grows at most polynomially; or $\sqrt{I} \ge Ce^{\frac{r^2}{4}} r^{-(n + 2\lambda - \epsilon)}$, $u$ grows at least like $Ce^{\frac{r^2}{4}} r^{-(n + 2\lambda - \epsilon)}$. They use $\sqrt{I}$ to describe the asymptotic behavior of $u$. In this paper, we give the expansion of $u$ in terms of spherical harmonics in Lemma 3.2 to see its asymptotic behavior directly.

To know the sharp dichotomy phenomenon of the growth rate, we compare the expansion of quasi-harmonic function in Lemma 2.3 with the expansion harmonic function. 

Assume $\varphi_k(\theta)$ is an eigenfunction of $L^2(S^{m-1})$ corresponding to the eigenvalue $\lambda_k$ and $C(N)$ is an Euclidean Cone $(0,\infty)\times N^{m-1}.$  
\begin{theorem}\label {thm:1.5}
\cite{10} If $u$ is a harmonic function on $C(N)$, then
\[ u(r,\theta) = \sum c_k r^{p_k} \varphi_k(\theta),\]
where the $c_k$ are constants and $p_k = \frac{-(m-2) + \sqrt{(m-2)^2 + 4\lambda_k}}{4}$ increases strictly from $0$ to $+\infty$ as $k\rightarrow +\infty$. Furthermore, $u$ has polynomial growth if and only if this is a finite sum.
\end{theorem}
We know that the property of quasi-harmonic  and harmonic functions are quiet different.
Let us recall some basic results about the metric $g=e^{-\frac{|x|^2}{2(m-2)}}ds_0^2.$ Lin-Wang \cite{1}  introduced the quasi-harmonic sphere, which is a harmonic map from $M=\big(\mathbb{R}^m, e^{-\frac{|x|^2}{2(m-2)}}ds_0^2\big)$ to $N$ with finite energy when they study the regularity of the heat flow of harmonic maps(c.f.\cite{12}). Here $ds_0^2$ is Euclidean metric in $\mathbb{R}^m$.  Colding-Minicozzi \cite{11} also pointed out  self-shrinkers in $\mathbb{R}^{m-1}$ are minimal hypersurfaces for the metric $g=e^{-\frac{|x|^2}{2(m-2)}}ds_0^2.$ Note that $\mathbb{R}^m$ with this metric is actually a  Riemannian manifold with a singularity at $\infty.$ The compactification of $\mathbb{R}^m$ provided by this metric is a topological $m-$sphere. Colding-Minicozzi \cite{11} mentioned that the Ricci curvature of this metrics does not have a sign and goes to negative infinity at infinity and thus there is no way to smoothly extend the metric to a neighborhood of infinity. The metric $g=e^{-\frac{|x|^2}{2(m-2)}}ds_0^2$ is quite singular at infinity and it is not complete.  Ding-Zhao \cite{2} showed that if the target $N$ is a sphere, any equivariant quasi-harmonic sphere is discontinuous at infinity and conjectured that any non-constant quasi-harmonic sphere is discontinuous at infinity. 

In this paper, we will give a more precise description of the behavior of quasi-harmonic function and eigenfunctions of $\Delta_h$ near the infinity.

Assume $u_0 = \frac{1}{2} \int_{0}^{r} e^{\frac{r^2}{4}} r^{1-m} d r,$ which is an radially symmetric solution of $\Delta_g u = 0$, we will show that $\frac{u(r,\theta)}{u_0(r)}$ could be asymptotic to any given function $g(\theta) \in  H^{[\frac{m}{2}]+2}(S^{m-1})$ from the following result.
\begin{theorem}\label {thm:1.6}
Assume $\bar{M} = (\mathbb{R}^m, g)$, for any given function $g(\theta) \in  H^{[\frac{m}{2}]+2}(S^{m-1})$, there exists a quasi-harmonic function $u(r,\theta)$ (i.e., $\Delta_g u = 0$) on $\mathbb{R}^m\backslash \{0\}$ such that $\lim\limits_{r \to +\infty} \frac{u(r,\theta)}{u_0(r)} = g(\theta)$. Moreover, if $\frac{u(r,\theta)}{u_0(r)}$ and $\frac{\bar{u}(r,\theta)}{u_0(r)}$ are asymptotic to the same function $g(\theta)$, then $u(r,\theta) = \bar{u}(r, \theta) + c.$
\end{theorem}

Similarly, assume $u_0 = e^{\frac{r^2}{4}} r^{-(m+2\lambda)}$, we have
\begin{theorem}\label {thm:1.7}
If $ 2\lambda$ is not an integer, for any given function $g(\theta) \in  H^{[\frac{m}{2}]+2}(S^{m-1})$, there exists an eigenfunction $u(r,\theta)$ of $\Delta_h$ on $\mathbb{R}^m\backslash \{0\}$ such that $\lim\limits_{r \to +\infty} \frac{u(r,\theta)}{u_0(r)} = g(\theta)$. Moreover, if $\frac{u(r,\theta)}{u_0(r)}$ and $\frac{\bar{u}(r,\theta)}{u_0(r)}$ are asymptotic to the same function $g(\theta)$, then $u(r,\theta) = \bar{u}(r, \theta)+p(r)$, where $p(r) \sim r^{2\lambda}$. If $\lambda$ is an integer, for any given function $g(\theta) \in  H^{[\frac{m}{2}]+2}(S^{m-1})$ satisfying that $\langle g, \varphi_k \rangle = 0$ when $k \in \{0, m-2, m, m+2, \cdots, m+2\lambda\} = A$. Then there exists an eigenfunction $u(r, \theta)$ of $\Delta_h$ on $\mathbb{R}^m\backslash \{0\}$ such that  $\lim\limits_{r \to +\infty} \frac{u(r,\theta)}{u_0(r)} = g(\theta)$. If $2\lambda$ is an integer and $\lambda$ is not an integer, for any given function $g(\theta) \in  H^{[\frac{m}{2}]+2}(S^{m-1})$ satisfying that $\langle g, \varphi_k \rangle = 0 $ when $k \in \{m-2, m, m+2, \cdots, m+2[\lambda]\} = B$. Then there exists an eigenfunction $u(r, \theta)$ of $\Delta_h$ on $\mathbb{R}^m\backslash \{0\}$ such that  $\lim\limits_{r \to +\infty} \frac{u(r,\theta)}{u_0(r)} = g(\theta)$.
\end{theorem}

Finally, we consider Liouville theorem of quasi-harmonic function with these expansions.One may wonder whether the quasi-harmonic functions still possess the basic properties of harmonic functions. Cheng-Yau \cite{3} proved that any harmonic function with sub-linear growth on manifolds with non-negative Ricci curvature must be constant. Li-Wang \cite{4} showed that there is no non-constant positive quasi-harmonic function on $\mathbb{R}^m$ with polynomial growth in Theorem 4.2. In this paper, we can improve this result by replacing the condition of polynomial growth with exponential growth.
 
\begin{theorem}\label {thm:1.1}
Let $u$ be a quasi-harmonic function (i.e, $\Delta_gu=0$) in $\mathbb{R}^m$. If there exists a sequence $r_i \rightarrow +\infty$ such that $\big(\int_{S^{m-1}} ( u(r_i,\theta))^2\big)^{\frac{1}{2}}  \le Ce^{\frac{r_i^2}{4}}r_i^{-(m+\epsilon)}$ for some $\epsilon > 0$, then $u$ is a constant.
\end{theorem}

Colding-Minicozzi  mentioned that if $\Delta_hu=0$ and $||u||_{L^2(\mathbb{R}^m)}^2=\int_{\mathbb{R}^m}u^2e^{-h} < \infty$ in \cite{6}, then $u$ must be constant.  More generally, they showed the following result in \cite{5}.
\begin{lemma}
\cite{5} If $\Delta_{h}u=-\lambda u$ on $\mathbb{R}^m$ and $\int_{\mathbb{R}^m}u^2e^{-h} < \infty$, then $\lambda$ is a half-integer and $u$ is a polynomial of degree $2\lambda$.
\end{lemma}
\begin{theorem}\label {thm:1.3}
If $\Delta_{h}u=-\lambda u$ on $\mathbb{R}^m\backslash \{0\}.$ If there exists a sequence  $r_i \rightarrow +\infty$ such that $\big(\int_{S^{m-1}} ( u(r_i,\theta))^2\big)^{\frac{1}{2}}  \le C e^{\frac{r_i^2}{4}} r_i^{-(m+2\lambda + \epsilon)}$ for some $\epsilon > 0$, then $u$ is a polynomial of degree $2 \lambda$ when $2 \lambda$ is an integer or $u=p(r)$ satisfying $p(r) \sim r^{2\lambda}$ when $2 \lambda$ is not an integer.
\end{theorem}
Since $\int_{\mathbb{R}^m}u^2e^{-h} =\int^{+\infty}_0\int_{S^{m-1}}u^2e^{-\frac{r^2}{4}}< \infty$ implies that there exists a sequence  $r_i \rightarrow +\infty$ such that $\big(\int_{S^{m-1}} ( u(r_i,\theta))^2\big)^{\frac{1}{2}} \le Ce^{\frac{r_i^2}{8}}$. Theorem 1.5 and Theorem 1.7 can also be see as a generalization of Lemma 1.6.

\hspace{0.4cm}
\section{The asymptotic behavior of quasi-harmonic functions at infinity}
\setcounter{equation}{0}
\medskip
Assume that $u$ is quasi-harmonic function , i.e.,
\begin{equation}
\Delta_g u = 0.
\end{equation}
We rewrite it in the following form
\begin{equation}
\Delta u - (\nabla h, \Delta u) = 0,
\end{equation}
where $h = \frac{r^2}{4}$. We know that the Euclidean metric of $\mathbb{R}^m$ can be written in spherical coordinates $(r, \theta)$ as
\[ ds^2 = dr^2 + r^2d\theta^2,\]
where $d\theta^2$ is the standard metric on $S^{m-1}$. Then
\[\Delta = \Delta_r + \frac{1}{r^2}\Delta_\theta,\]
where $\Delta_\theta$ is the Laplacian on the standard $S^{m-1}$. It is clear that
\[ \nabla h \cdot \nabla = \frac{r}{2} \frac{\partial}{\partial r}.\]
It follows from (2.2) that
\[ u_{rr} + \frac{m-1}{r} u_r + \frac{1}{r^2} \Delta_\theta u - \frac{r}{2} \frac{\partial u}{\partial r} = 0.\]

Let $\varphi_k$ be the orthonormal basis on $L^2(S^{m-1})$ corresponding to the eigenvalues,
\[ 0 = \lambda_0 < \lambda_1 \le \lambda_2 \le \cdots \le \lambda_k \to \infty \]
we have
\[ \Delta_\theta \varphi_k = - \lambda_k \varphi_k. \]
Let $\langle \cdot, \cdot \rangle$ denote $L^2$ inner product of $L^2 (S^{m-1})$. Then we have
\begin{align*}
\langle \Delta_r u, \varphi_k \rangle &= \Delta_r \langle u, \varphi_k \rangle,\\
\langle \Delta_\theta u, \varphi_k \rangle &= \langle u, \Delta_\theta \varphi_k \rangle = -\lambda_k \langle u, \varphi_k \rangle,\\
\langle \frac{\partial u}{\partial r}, \varphi _k \rangle &= \frac{\partial \langle u, \varphi_k \rangle}{\partial r}.
\end{align*}
Let $f_k(r) = \langle u(r,\cdot), \varphi_k \rangle$ for $k \ge 0$. Then we see that $f_k$ satisfies
\begin{equation}
(f_k)_{rr} + (\frac{m-1}{r} - \frac{r}{2})(f_k)_r = \frac{\lambda_k f_k}{r^2}.
\end{equation}
Let 
\begin{align*}
&f_k(r) = w(z)z^{l_k}, \\
&z = r^2, \\
&l_k = \frac{-(m-2) + \sqrt{(m-2)^2 + 4\lambda_k}}{4}.
\end{align*}
Then
\begin{align*}
(f_k)_r &= 2(w'(z)z^{l_k + \frac{1}{2}} + l_kz^{l_k - \frac{1}{2}}w(z)), \\
(f_k)_{rr} &= 4w''(z)z^{l_k+1} + 2(4l_k + 1)z^{l_k}w'(z) + l_k(4l_k - 2)z^{l_k -1}w(z).
\end{align*}
Then by (2.3), we have
\begin{equation}
zw''(z) + ((2l_k + \frac{m}{2}) - \frac{z}{4})w'(z) - \frac{l_k}{4} w(z) = 0.
\end{equation}
Assume
\[ w(z) = e^{\frac{z}{4}}y(-\frac{z}{4}).\]
Then
\begin{align*}
w_z &= \frac{1}{4} e^{\frac{z}{4}}(y(-\frac{z}{4}) - y'(-\frac{z}{4})), \\
w_{zz} &= \frac{1}{16} e^{\frac{z}{4}}(y(-\frac{z}{4}) - 2y'(-\frac{z}{4}) + y''(-\frac{z}{4})).
\end{align*}
By (2.4), we have
\begin{equation}
\frac{z}{4} y''(-\frac{z}{4}) + (-\frac{z}{4} - (2l_k + \frac{m}{2}))y'(-\frac{z}{4}) + (l_k + \frac{m}{2})y(-\frac{z}{4}) = 0.
\end{equation}
Let $x = -\frac{z}{4}$, then we have
\begin{equation}
xy''(x) + ((2l_k + \frac{m}{2}) -x)y'(x) - (l_k + \frac{m}{2})y(x) = 0.
\end{equation}
\begin{lemma}
For $k\ge 1,$ the general solution of (2.3) is
\[ f_k(r) = c_1 e^{\frac{r^2}{4}} r^{2l_k} F[l_k + \frac{m}{2}, 2l_k + \frac{m}{2}; -\frac{r^2}{4}],\]
where $F[a,b;x]$ is a Kummer's function in \cite{9} (see Page 2).
\end{lemma}
\begin{proof}
Set $b = 2l_k + \frac{m}{2}, a = l_k + \frac{m}{2}$. Assume that one solution of the Kummer's equation (2.6) is
\[ y = a_0 x^c + a_1 x^{c+1} + a_2 x^{c+2} + \cdots + a_n x^{c+n} + \cdots .\]
If we substitute this series and its first two derivatives in the differential equation, and then equate to zero the coefficients of powers of $x$, we find  that
\begin{align*}
a_0 c(c + b - 1) &= 0, \\
a_1 (c + 1)(c + b) &= a_0(c+a), \\
a_2(c+2)(c + b + 1) &= a_1 (c + a + 1), \\
&\cdots
\end{align*}
Since $a_0\ne 0,$ it has two roots:\\
(i) $c=0$, we have
\[ a_n = \frac{a\cdots (a+n-1)}{b\cdots (b+n-1)n!},\]
which gives one solution in terms of Kummer's series
\[ y = a_0 F [a,b;x].\]
(ii) $c = 1 - b$, which leads to a second solution
\begin{align*}
a_1 (2 - b)& = a_0 (1 - b + a), \\
&\cdots \\
a_{n-1} (n - b)(n-1) &= a_{n-2} (1-b+a+n-2), \\
a_n (n + 1 - b)n & = a_{n-1} (1 - b + a + n -1), \\
a_{n+1}(n+2-b)(n + 1)& = a_n(1-b+a+n), \\
&\cdots 
\end{align*}

If $b = n $ for some integer $n$, since $b > a > 1$, then $1 - b + a \le 0$ and $1 - b + a +n - 2 > 0$. It implies that the solution is a polynomial of some degree $i\le n$
\[ y = x^{1-b} (a_0 + a_1x + \cdots + a_i x^i).\]
If $b \neq n $ for any integer $n > 0$, then
\[a_k = \frac{(1-b+a)\cdots(1-b+a + k -2)}{((k+1)-b)\cdots (2-b)k!},\]
and
\[y = x^{1-b}(a_0 + a_1x + \cdots + a_k x^k+ \cdots).\]
Then the solution of (2.3) is
\begin{align*}
f_k(r) &= c_1 e^{\frac{r^2}{4}} r^{2l_k} y(-\frac{r^2}{4}) \\
&= c_1 e^{\frac{r^2}{4}} r^{2l_k} (-\frac{r^2}{4})^{1-(2l_k + \frac{m}{2})}(a_0 + a_1(-\frac{r^2}{4}) + \cdots) \\
&= (-4)^{2l_k + \frac{m}{2} - 1} c_1 e^{\frac{r^2}{4}} r^{2 - (2l_k + m)}(a_0 + a_1(-\frac{r^2}{4}) + \cdots),
\end{align*}
which implies that $\lim\limits_{r \to 0} f_k(r) \neq 0$. This contradicts with the initial condition that
\[ f_k(0) = \lim_{r \to 0} f_k(r) = \langle \lim_{r \to 0} u(r,\theta), \varphi_k \rangle = \langle u(0), \varphi_k \rangle = 0.\]
Hence the second solution should be ruled out and the solution of (2.3) has the
following form
\[ f_k(r) = c_k e^{\frac{r^2}{4}} r^{2l_k} F[l_k + \frac{m}{2}, 2l_k + \frac{m}{2}; -\frac{r^2}{4}].\]
\end{proof}

\begin{lemma}

For any fixed $\delta>0$ and  $k\ge 1,$ we set $a=l_k + \frac{m}{2}, b=2l_k + \frac{m}{2}, c=-a+\delta$. Assume $L> \delta$ and $x$ is a positive real number, the following
asymptotic relation holds:
\[ F[a, b; -x] = x^{-a} \frac{\Gamma(b)}{\Gamma(b-a)}\Big(1 +\sum_{n=1}^{[L-\delta]} \frac{(a)_n(1+a-b)_n}{n!} x^{-n} + \frac{\Gamma(b-a)}{\Gamma(a)}J_k(x)x^a\Big), \]
 where $(a)_n=a(a+1)\cdots(a+n-1)$ and $J_k(x)= \int_{c-i\infty}^{c+i\infty} \frac{\Gamma(L-s)\Gamma(a - L +s)}{2\pi i \Gamma(b-L+s)} x^{s-L}ds$. Moreover, $\left| J_k(x)\right|\le C(m,L,\delta)l_k^{2(L-\delta)+\frac{m}{2}} |x|^{c-L}$ for $k$ sufficient large.
Assume $g_k(x)=\sum_{n=1}^{[L-\delta]} \frac{(a)_n(1+a-b)_n}{n!} x^{-n} + \frac{\Gamma(b-a)}{\Gamma(a)}J_k(x)x^a$, in particular, if $k$ is sufficient large, we set $L=2\delta,$ then
\[ |g_k(x)| \le C(m,\delta)\frac{l_k^{2\delta}}{x^{\delta}},\]  
 and if $k\le K$ for some $K>0$, we set $L=\delta+1,$ then 
\[|g_k(x)| \le C(m,K,\delta)\frac{1}{x}.\]

\end{lemma}

\begin{proof}
Using the results in \cite{9} (see Page 36), we have
\begin{equation}
|z^s| = |z|^{\text{Re}(s)} e^{-\text{Im}(s)\arg{z}},
\end{equation}
and
\begin{equation}
\frac{\Gamma (a)}{\Gamma (b)} F[a, b ; -x] = \frac{1}{2\pi i}\int_{c-i\infty}^{c + i\infty}\frac{\Gamma(-s)\Gamma(a+s)}{\Gamma(b+s)} x^s ds,
\end{equation}
provided that $|\arg{x}| < \frac{1}{2} \pi$ and $b \neq 0, -1, -2, \cdots$. Now we will deduce the asymptotic expansion in $x$ for Kummer's function. Let us consider the integral
\[ I_k = \frac{1}{2\pi i} \int_{ADEF} \frac{\Gamma(-s)\Gamma(a+s)}{\Gamma(b+s)} x^s ds\]
round the rectangular contour
\[ A(c - iN),\ D(c + iM), \ E(c - L + iM),\  F(c - L - iN)\]
in the $s-$ space.
\begin{align*}
I_k &= \int_{AD} + \int_{DE} + \int_{EF} + \int_{FA} \\
&=I_{M,N} + J_4 + J_5 + J_6.
\end{align*}
As $M$ and $N \rightarrow \infty,$
\begin{align*}
&J_4 \to 0, \\
&J_6 \to 0, \\
&-J_5 \to J_k = \int_{c-i\infty}^{c+i\infty} \frac{\Gamma(L-s)\Gamma(a - L +s)}{2\pi i \Gamma(b-L+s)} x^{s-L}ds, \\
&I_{M,N} \to I_1 = \int_{c-i\infty}^{c+i\infty} \frac{\Gamma(-s)\Gamma(a + s)}{2\pi i \Gamma(b+s)} x^s ds \\
&= \frac{\Gamma(a)}{\Gamma(b)}F[a,b;-x],
\end{align*}
and
\[ I_k = -\sum_{n=0}^{[L]} \frac{\Gamma(a+n)}{\Gamma(b-a-n)} \frac{(-1)^{n-1}}{n!} x^{-a-n},\]
from the residues at the poles $s=-a, -a-1, \cdots, -a - [L - \delta]$. Then we have
\begin{align*}
F[a,b;x] &= \frac{\Gamma (b)}{\Gamma (a)}(I_k + J_k) \\
&= x^{-a} \frac{\Gamma(b)}{\Gamma(b-a)}(\sum_{n=0}^{[L-\delta]} \frac{(a)_n(1+a-b)_n}{n!} x^{-n} + \frac{\Gamma(b-a)}{\Gamma(a)}J_k(x)x^a).
\end{align*}
Assume $s= it + c$, then
\[ J_k = \int_{-\infty}^{+\infty} \frac{\Gamma(L-it-c)\Gamma(a-L+it+c)}{2\pi i \Gamma(b-L+it+c)}x^{it+c-L} idt.\]
By (2.7), we have
\[|x^{it+c-L}| = |x|^{c-L} e^{-(\arg{x})t}, \]
if $x$ is a positive real number, we have
\[|x^{it+c-L}| = x^{c-L} . \]
We know that the classical stirling's approximate formula for the Gamma-Function
in the form
\[ \Gamma(z) = \sqrt{2\pi} e^{-z} z^{z-\frac{1}{2}}(1 + O(\frac{1}{|z|})),\]
for $|\arg{z}| < \pi$ as $|z| \to \infty$. In absolute value:
\[ |\Gamma(x + iy)| = \sqrt{2\pi}e^{-x}(x^2 + y^2)^{\frac{x-\frac{1}{2}}{2}}e^{-y\arg{(x+iy)}}, \ as \ \sqrt{x^2 + y^2} \to +\infty. \]
If we set $a = l_k + \frac{m}{2}$, $b = 2l_k + \frac{m}{2}$ and $ c = -a + \delta,$ then\begin{align*}
&L-c = l_k + \frac{m}{2} + L - \delta, \\
&a +c -L = \delta -L,\\
&b + c - L = l_k + \delta - L.
\end{align*}
Thus
\begin{align*}
|\Gamma(L - it -c)| &\sim ((L -c)^2 + t^2)^{\frac{L-c-\frac{1}{2}}{2}} \exp({t \arg{(L-it-c)}}) e^{c-L}, \\
|\Gamma(a -L + it + c)| &\sim ((a + c - L)^2 + t^2)^{\frac{a-L-c-\frac{1}{2}}{2}} \\
&\exp(-t\arg{(a - L + c+ it)})\exp(-(a-L+c)),\\
|\Gamma(b -L + it + c)| &\sim ((b - L + c)^2 + t^2)^{\frac{b-L+ c-\frac{1}{2}}{2}} \\
&\exp(-t\arg{(b - L + c+ it)}) \exp(-(b-L+c)).
\end{align*}
If $k$ is large enough, we have
\begin{align*}
&\frac{((L-c)^2+t^2)^{\frac{L-c-\frac{1}{2}}{2}}((a+c -L)^2 +  t^2)^{\frac{a-L+c-\frac{1}{2}}{2}}}{((b - L +c)^2 + t^2 )^{\frac{b-L+c - \frac{1}{2}}{2}}} \\
&=\frac{((L-\delta + l_k + \frac{m}{2})^2 + t^2)^{\frac{l_k + \frac{m}{2} + L -\delta -\frac{1}{2}}{2}}((L-\delta)^2 + t^2)^{\frac{\delta - L - \frac{1}{2}}{2}}}{((l_k +\delta - L)^2 + t^2)^{\frac{l_k - L +\delta - \frac{1}{2}}{2}}} \\
&=(1 + \frac{(m + 4(L-\delta))(l_k - L + \delta) + (\frac{m}{2} + 2(L-\delta))^2}{t^2 + (l_k - L +\delta)^2})^{\frac{l_k - L + \delta - \frac{1}{2}}{2}} \\
&((l_k + \frac{m}{2} + L - \delta)^2 + t^2)^{\frac{2(L-\delta) + \frac{m}{2}}{2}}((L-\delta)^2 + t^2)^{\frac{\delta -L - \frac{1}{2}}{2}} \\
&\le \frac{\exp(\frac{m}{2} + 2(L - \delta))}{(L-\delta)^{L-\delta}}((l_k + \frac{m}{2} + L -\delta)^2 + t^2)^{\frac{2(L-\delta) + \frac{m}{2}}{2}}\\
&\le C(m,L,\delta)((l_k + \frac{m}{2} + L -\delta)^2 + t^2)^{\frac{2(L-\delta) + \frac{m}{2}}{2}}.
\end{align*}
For $t > 0$, we have
\begin{align*}
&\frac{\exp(t\arg(L-it-c))\exp(-t\arg(a-L+c+it))}{\exp(-t\arg(b-L+c+it)} \\
&=\exp\big(-t \arctan \frac{t}{L-c} - t(\pi + \arctan \frac{t}{a-L+c}) + t\arctan \frac{t}{b-L+c}\big) \\
\end{align*}
\begin{align*}
&=\exp(t \arctan \frac{\frac{t}{b-L+c} - \frac{t}{L-c} }{1+ \frac{t^2}{(b-L+c)(L-c)}})\exp(-t(\pi + \arctan \frac{t}{a-L+c}))\\
&\le \exp(t\arctan \frac{2L-2c -b}{t + \frac{(b-L+c)(L-c)}{t} })\exp(-\frac{\pi}{2}t) \\
&= \exp(t\arctan \frac{2L-2\delta +\frac{m}{2}}{t + \frac{(l_k + \delta -L)(l_k + \frac{m}{2} + L -\delta)}{t} })\exp(-\frac{\pi}{2}t) \\
&\le \exp(-\frac{\pi}{4}t).
\end{align*}
For $t < 0$, we similarly have
\begin{align*}
&\frac{\exp(t\arg(L-it-c)) \exp(-t \arg(a-L+c+it))}{\exp(-t \arg(b-L+c+it)} \\
&=\exp(-t \arctan \frac{t}{L-c} -t(\arctan \frac{t}{a-L+c} - \pi) + t\arctan \frac{t}{b-L+c} \\
&\le \exp(\frac{\pi}{4}t).
\end{align*}
And 
\begin{align*}
&\exp (c-L) \exp(-(a-L+c)) \exp(b-L+c)\\
&=\exp(-a+b-L+c) \\
&=\exp(-\frac{m}{2} - L + \delta).
\end{align*}
Then we can get
\begin{align*}
&|J_k| \le \int_{-\infty}^{+\infty}\left |\frac{\Gamma(L-it-c)\Gamma(a-L+it+c)}{2\pi i\Gamma(b-L+it+c)} x^{c-L}\right |dt \\
&=C(m,L,\delta) \int_{0}^{+\infty} \big((l_k + \frac{m}{2} + L -\delta)^2 + t^2\big)^{\frac{2(L-\delta)+\frac{m}{2}}{2} } \exp(-\frac{\pi}{4}t)dtx^{c-L} \\
&\le C(m,L,\delta)l_k^{2(L-\delta)+\frac{m}{2}} x^{c-L}.
\end{align*}
Since
\[ \frac{\Gamma(b-a)}{\Gamma(a)} \sim \frac{\sqrt{2\pi l_k}(\frac{l_k}{e})^{l_k}}{\sqrt{2\pi(l_k + \frac{m}{2})}(\frac{l_k + \frac{m}{2}}{e})^{l_k + \frac{m}{2}}},\]
then
\[ \left | \frac{\Gamma(b-a)}{\Gamma(a)} \right| \le \frac{C(m)}{(l_k + \frac{m}{2})^{\frac{m}{2}}}.\]
We set $L=2\delta$, then
\begin{align*}
&\left | \frac{\Gamma(b-a)}{\Gamma (a)} J_k(x) x^a \right| \\
& \le C(m,\delta) l_k^{2\delta} x^{-\delta}.
\end{align*}
Thus
\begin{align*}
|g_k(x)| &= \left | \sum_{n=1}^{[L-\delta]} \frac{a_n(1+a-b)_n}{n!}x^{-n} + \frac{\Gamma(b-a)}{\Gamma(a)}J_k(x) x^a \right | \\
&\le C(m,\delta) l_k^{2\delta}x^{-\delta}.
\end{align*}
If $k\le K$ for some $K>0$, $a,b,c$ are all bounded. We set $L=\delta+1$,  then
\[|g_k(x)| \le C(m,K,\delta)\frac{1}{x}.\]
\end{proof}
If $k\ge 1,$ we can use the asymptotic relation in Lemma 2.2, then

\begin{align*}
 f_k(r)& = c_k e^{\frac{r^2}{4}} r^{2l_k} F[l_k + \frac{m}{2}, 2l_k + \frac{m}{2}; -\frac{r^2}{4}]\\
&=c_k\frac{\Gamma(2l_k + \frac{m}{2})}{\Gamma(l_k)}e^{\frac{r^2}{4}}r^{-m}\Big(1+g_k(\frac{r^2}{4})\Big) \\
&= C_ke^{\frac{r^2}{4}}r^{-m}\Big(1+g_k(\frac{r^2}{4})\Big),
\end{align*}
where $C_k$ are constants.
If $k=0,$ the asymptotic relation of Kummer's function does not hold, we solves the equation directly.
\begin{equation}
(f_0)_{rr} + (\frac{m-1}{r} - \frac{r}{2})(f_0)_r = 0.
\end{equation}
Then
\begin{equation}
f_0 = c + C_0\int_0^r e^{\frac{r^2}{4}} r^{-m} dr
\end{equation}
is the general solution of (2.4).
Then we can easily get that:
\begin{lemma}
The general solution of $\Delta_g u=0$ on $\mathbb{R}^{m}\backslash \{0\}$ has the following form:
\[u(r,\theta)=c + C_0\int_0^r e^{\frac{r^2}{4}} r^{-m} dr+\sum_{k=1}^{\infty}C_ke^{\frac{r^2}{4}}r^{-m}\Big(1+g_k(\frac{r^2}{4})\Big)\varphi_k.\]
\end{lemma}

\begin{lemma}
Assume $u\in H^{2n}(S^{m-1}),$ the fourier coefficients $|\langle u, \varphi_k\rangle|\le \frac{C}{\lambda_k^n},$ where $C$ is independent of $k.$
\end{lemma}

\begin{proof}
\begin{align*}
\langle u,\varphi_k \rangle &=\int_{S^{m-1}}u\varphi_k\\
=&-\frac{1}{\lambda_k}\int_{S^{m-1}}u\Delta\varphi_k\\
=&-\frac{1}{\lambda_k}\int_{S^{m-1}}\Delta u\varphi_k\\
=&\frac{1}{\lambda_k^2}\int_{S^{m-1}}\Delta u\Delta \varphi_k\\
=&\frac{1}{\lambda_k^2}\int_{S^{m-1}}\Delta^2 u \varphi_k\\
&\vdots \\
&=(-\frac{1}{\lambda_k})^n\int_{S^{m-1}}\Delta^n u \varphi_k.
\end{align*}
\[|\langle u, \varphi_k\rangle| \le \frac{1}{\lambda_k^n}\int_{S^{m-1}}(\Delta^n u)^2\le \frac{C}{\lambda_k^n}.\]
\end{proof}
 \begin{proof}[Proof of Theorem \ref{thm:1.1}]
For any fixed $k\ge 1,$ note that 
\begin{align*}
\frac{1}{2}|C_k|e^{\frac{r_i^2}{4}}r_i^{-m}& \le |C_k||1+g_k(\frac{r_i^2}{4})|e^{\frac{r_i^2}{4}}r_i^{-m}\\
&=|\int_{S^{m-1}}u(r_i,\theta)\varphi_k|\\
&\le \big(\int_{S^{m-1}} ( u(r_i,\theta))^2\big)^{\frac{1}{2}} (\int_{S^{m-1}} \varphi_k^2)^{\frac{1}{2}}\\
&\le C\sqrt{\omega_n}e^{\frac{r_i^2}{4}}r_i^{-(m+\epsilon)},
\end{align*}
 which implies that 
\[C_k<Cr_i^{-\epsilon}.\]
Let $r_i \rightarrow +\infty$, we have
\[C_k\equiv 0.\]
If $k=0,$ 
\[|c+\frac{1}{2}C_0\int_0^{r_i} e^{\frac{r^2}{4}}r^{1-m}dr|\le C\sqrt{\omega_n}e^{\frac{r_i^2}{4}}r_i^{-(m+\epsilon)},\]
which implies that 
\[C_0\equiv 0.\]
It then follows that 
\[u\equiv c.\]
\end{proof}

 \begin{proof}[Proof of Theorem \ref{thm:1.6}]
For any $g(\theta)\in H^{[\frac{m}{2}]+2}(S^{m-1}),$ by Lemma 2.4, we have $g(\theta)=\sum_{k=0}^{\infty} \bar C_k\varphi_k$ and $|\bar C_k|\le \frac{C}{\lambda_k^{\frac{1}{2}[\frac{m}{2}]+1}}.$ By Lemma 2.1, for any $k\ge 1,$
\[\bar f_k= \bar C_ke^{\frac{r^2}{4}}r^{-m}\Big(1+g_k(\frac{r^2}{4})\Big) \]
is a solution of (2.3). \\
Similarly, 
\[\bar f_0(r)=\frac{1}{2}\bar C_0\int_0^re^{\frac{r^2}{4}}r^{-m}dr\] 
is a solution of (2.9). Assume $L=2\delta,$ by using the fact that $|\varphi_k|^2\le C(m)k^{m-2},$ we have 
\[|\bar f_k \varphi_k|\le Ce^{\frac{r^2}{4}}r^{-m}\frac{1}{k^{2}}\Big(1+C(m,K,\delta)\frac{k^{2\delta}}{r^{2\delta}}\Big)\]
for $k$ bounded by some $K>0$, where $C$ is independent of $k$. Assume $L=1+\delta,$ similarly we have
\[|\bar f_k \varphi_k|\le Ce^{\frac{r^2}{4}}r^{-m}\frac{1}{k^{2}}\Big(1+C(m,\delta)\frac{1}{r^{2}}\Big)\]
for $k$ finite, where $C$ is independent of $k$.
If we set $0<\delta<\frac{1}{2},$ the series $\sum_{k=0}^{\infty} \bar f_k\varphi_k$ is uniformly convergent on $\overline {B_{A}(0)} \setminus B_{\epsilon}(0)$ for any $\epsilon,A>0$, where $B_{\epsilon}(0)$ is an open ball. Assume $\bar u=\sum_{k=0}^{\infty} \bar f_k\varphi_k,$ then we have 
\begin{align*}
\Delta_g(\bar u)&=\bar u_{rr}+\frac{m-1}{r}\bar u_r+\frac{1}{r^2}\Delta_{\theta}\bar u-\frac{r}{2}\frac{\partial \bar u}{\partial r}\\
&=\sum_{k=1}^{\infty} \Big((\bar f_k)_{rr}+(\frac{m-1}{r}-\frac{r}{2})(\bar f_k)_r-\frac{\lambda_k\bar f_k}{r^2}\Big)\varphi_k+(\bar f_0)_{rr}+(\frac{m-1}{r}-\frac{r}{2})(\bar f_0)_r\\
&=0.
\end{align*}
Thus $\bar u$ is a solution of $\Delta_g u=0$ on $\mathbb{R}^m\backslash \{0\}.$\\
Since $\sum_{k=0}^{\infty} \frac{\bar f_k}{u_0}\varphi_k$ is uniformly convergent on $\mathbb{R}^m\setminus B_{\epsilon}(0)$, we can obtain that 
\begin{align*}
\lim\limits_{r\to+\infty}\frac{\bar u(r,\theta)}{u_0(r)}=\lim\limits_{r\to+\infty}\sum_{k=0}^{\infty} \frac{\bar f_k}{u_0}\varphi_k=\sum_{k=0}^{\infty}\lim\limits_{r\to + \infty}\frac{\bar f_k}{u_0}\varphi_k=\sum_{k=0}^{\infty} \bar C_k\varphi_k=g(\theta).
\end{align*}
Hence $\bar u$ is the quasi-harmonic function which satisfies that $\lim\limits_{r\to+\infty}\frac{\bar u(r,\theta)}{u_0(r)}=g(\theta)$ for the given function $g(\theta)\in H^{[\frac{m}{2}]+2}(S^{m-1}).$
Moreover, assume $\frac{\bar u(r,\theta)}{u_0(r)}$ and $\frac{ u(r,\theta)}{u_0(r)}$ are asymptotic to the same function $g(\theta)$, for $k\ge 1,$ we have
\begin{align*}
&C_k=\lim\limits_{r\to+\infty}\frac{ f_k}{u_0}=\lim\limits_{r\to+\infty}\langle \frac{ u(r,\theta)}{u_0(r)}, \varphi_k \rangle \\
&=\langle \lim\limits_{r\to+\infty}\frac{\bar u(r,\theta)}{u_0(r)}, \varphi_k \rangle=\lim\limits_{r\to + \infty}\frac{\bar f_k}{u_0}=\bar C_k.
\end{align*}
By (2.5), we also have 
\[C_0=\bar C_0.\]
Hence 
\begin{align*}
&u(r,\theta)-\bar u(r, \theta)\\
&=c-\bar c+(C_0-\bar C_0)\int^r_0e^{\frac{r^2}{4}}r^{1-m}dr +\sum_{k=0}^{\infty} (C_k-\bar C_k)e^{\frac{r^2}{4}}r^{-m}\Big(1+g_k(\frac{r^2}{4})\Big)\\
&=c-\bar c.
\end{align*}
\end{proof}

\hspace{0.4cm}
\section{The asymptotic behavior of eigenfunctions of the drift laplacian at infinity}
\setcounter{equation}{0}
\medskip
Assume that $u$ is an eigenfunction of $\Delta_h$ , i.e.,
\begin{equation}
\Delta_h u = -\lambda u.
\end{equation}
We rewrite it in the following form
\begin{equation}
\Delta u - (\nabla h, \Delta u) = -\lambda u,
\end{equation}
where $h = \frac{r^2}{4}$. 
Let $f_k(r)=\langle u(r,\cdot),\varphi_k\rangle$ for $k\ge 0,$ we have
\begin{equation}
(f_k)_{rr} + (\frac{m-1}{r} - \frac{r}{2})(f_k)_r = (\frac{\lambda_k }{r^2}-\lambda)f_k.
\end{equation}
Let 
\begin{align*}
&f_k(r) = w(z)z^{l_k} ,\\
&z = r^2 ,\\
&l_k = \frac{-(m-2) + \sqrt{(m-2)^2 + 4\lambda_k}}{4}.
\end{align*}
Then by (3.3), we have
\begin{equation}
zw''(z) + ((2l_k + \frac{m}{2}) - \frac{z}{4})w'(z) - \frac{l_k-\lambda}{4} w(z) = 0.
\end{equation}
Assume
\[ w(z) = e^{\frac{z}{4}}y(-\frac{z}{4}).\]
Then by (3.4), we have
\begin{equation}
\frac{z}{4} y''(-\frac{z}{4}) + (-\frac{z}{4} - (2l_k + \frac{m}{2}))y'(-\frac{z}{4}) + (l_k + \frac{m}{2}+\lambda)y(-\frac{z}{4}) = 0.
\end{equation}
Let $x = -\frac{z}{4}$, then we have
\begin{equation}
xy''(x) + ((2l_k + \frac{m}{2}) -x)y'(x) - (l_k + \frac{m}{2}+\lambda)y(x) = 0.
\end{equation}
Set $b = 2l_k + \frac{m}{2}, a = l_k + \frac{m}{2}+\lambda$.  The first  solution of (3.6) in terms of Kummer's series is 
\[y=c_kF[a,b;x].\]
The first solution of (3.3) is 
\[f_k(r)=c_ke^{\frac{r^2}{4}}r^{2l_k}F[ l_k + \frac{m}{2}+\lambda,2l_k + \frac{m}{2};-\frac{r^2}{4}].\]
If $k\ge 1,$ the second solution of (3.3) should be ruled out with the same reason as in the proof of Lemma 2.1.
We can use the same method as in the proof of Lemma 2.2 to get the asymptotic relation of the Kummer's functions.

\begin{lemma}
For any fixed $\delta>0$ and  $k$ satisfying that  $l_k-\lambda\neq-i$ for any nonnegative integer $i$, i.e., $\frac{1}{\Gamma(l_k-\lambda)}\neq 0,$ we set $a=l_k + \frac{m}{2}+\lambda, b=2l_k + \frac{m}{2}, c=-a+\delta$. Assume $L> \delta$ and $x$ is a positive real number, the following
asymptotic relation holds:
\[ F[a, b; -x] = x^{-a} \frac{\Gamma(b)}{\Gamma(b-a)}\Big(1 +\sum_{n=1}^{[L-\delta]} \frac{a_n(1+a-b)_n}{n!} x^{-n} + \frac{\Gamma(b-a)}{\Gamma(a)}J_k(x)x^a\Big), \]
 where $(a)_n=a(a+1)\cdots(a+n-1)$ and $J_k(x)= \int_{c-i\infty}^{c+i\infty} \frac{\Gamma(L-s)\Gamma(a - L +s)}{2\pi i \Gamma(b-L+s)} x^{s-L}ds$. Moreover, $\left| J_k(x)\right|\le C(m,L,\delta,\lambda)l_k^{2(L-\delta)+\frac{m}{2}} |x|^{c-L}$ for $k$ sufficient large. Assume $g_k(x)=\sum_{n=1}^{[L-\delta]} \frac{(a)_n(1+a-b)_n}{n!} x^{-n} + \frac{\Gamma(b-a)}{\Gamma(a)}J_k(x)x^a$, in particular, if $k$ is sufficient large, we set $L=2\delta$, then
\[ |g_k(x)| \le C(m,\delta,\lambda)\frac{l_k^{2\delta}}{x^{\delta}}\]  
and if $k\le K$ for some $K>0$, we set $L=\delta+1,$ then
\[ |g_k(x)| \le C(m,\delta,\lambda,K)\frac{1}{x}.\]  
\end{lemma}

Note that  the points of nonpositive integer are the zero points of $\frac{1}{\Gamma(z)}.$ When $l_k-\lambda=-i$ for any nonnegative integer $i,$ the Kummer's functions $F[l_k+\frac{m}{2}+\lambda,2l_k+\frac{m}{2};-x]$ do not have the same asymptotic relation as in Lemma 3.1.  We will deal with the Kummer's series directly.
Assume $l_k-\lambda=-i$ for some nonnegative integer $0\le i\le \lambda.$ Then the first solution of (3.6) is
\begin{align*}
y(x)&=\sum_{n=0}^{\infty}\frac{(l_k+\frac{m}{2}+\lambda)\cdots(l_k+\frac{m}{2}+\lambda+n-1)}{(l_k+\frac{m}{2}+\lambda-i)\cdots(l_k+\frac{m}{2}+n+1+\lambda-i)n!}x^n\\
&=\sum_{n=0}^{\infty}\frac{(l_k+\frac{m}{2}+\lambda-i+1+n-1)\cdots(l_k+\frac{m}{2}+\lambda+n-1)}{(l_k+\frac{m}{2}+\lambda-i)\cdots(l_k+\frac{m}{2}+\lambda-1)n!}x^n
\end{align*}
Multiply $x^{l_k+\frac{m}{2}+\lambda-i-1}$ on both sides, we have
\begin{align*}
&y(x)x^{l_k+\frac{m}{2}+\lambda-i-1}=\\
&\sum_{n=0}^{\infty}\frac{(l_k+\frac{m}{2}+\lambda-i+1+n-1)\cdots(l_k+\frac{m}{2}+\lambda+n-1)}{(l_k+\frac{m}{2}+\lambda-i)\cdots(l_k+\frac{m}{2}+\lambda-1)n!}x^{n+l_k+\frac{m}{2}+\lambda-i-1}.
\end{align*}
Integrating from $0$ to $x$ for $i$ times, then we have 
\[\int_0^x\cdots \int_0^x y(x)x^{l_k+\frac{m}{2}+\lambda-i-1}=\sum_{n=0}^{\infty}\frac{x^{n+l_k+\frac{m}{2}+\lambda-1}}{(l_k+\frac{m}{2}+\lambda-i)\cdots(l_k+\frac{m}{2}+\lambda-1)n!}.\]
Thus 
\[\int_0^x\cdots \int_0^x y(x)x^{l_k+\frac{m}{2}+\lambda-i-1}=\frac{x^{l_k+\frac{m}{2}+\lambda-1}}{(l_k+\frac{m}{2}+\lambda-i)\cdots(l_k+\frac{m}{2}+\lambda-1)}\sum_{n=0}^{\infty}\frac{x^{n}}{n!}.\]
Differentiate the equation on both sides for $i$ times, we have
\[y(x)x^{l_k+\frac{m}{2}+\lambda-i-1}=e^x\Big(x^{l_k+\frac{m}{2}+\lambda-i-1}+\cdots+\frac{x^{l_k+\frac{m}{2}+\lambda-1}}{(l_k+\frac{m}{2}+\lambda-i)\cdots(l_k+\frac{m}{2}+\lambda-1)}\Big).\]
Then 
\[y(x)=e^x\Big(1+\cdots+\frac{x^{i}}{(l_k+\frac{m}{2}+\lambda-i)\cdots(l_k+\frac{m}{2}+\lambda-1)}\Big).\]
Thus 
\begin{align*}
(f_k)_1(r)&=e^{\frac{r^2}{4}}r^{2l_k}y(-\frac{r^2}{4})\\
&=c_k\Big(r^{2l_k}+\cdots+(\frac{1}{4})^{2i}\frac{r^{2\lambda}}{(l_k+\frac{m}{2}+\lambda-i)\cdots(l_k+\frac{m}{2}+\lambda-1)}\Big)\\
&=c_k q_k(r).
\end{align*}
Here $q_k(r)=r^{2l_k}+\cdots+(\frac{1}{4})^{2(l_k-\lambda)}\frac{r^{2\lambda}}{(2l_k+\frac{m}{2})\cdots(l_k+\frac{m}{2}+\lambda-1)}$ is a polynomial of degree $2\lambda.$
\begin{lemma}
Assume $A=\{0,m-2,m,m+2,\cdots,m+2\lambda\}$ and $B=\{m-2,m,m+2,\cdots,m+2[\lambda]\}.$ The general solution of $\Delta_h u=-\lambda u$ on $\mathbb{R}^m\backslash \{0\}$ has the following form:\\
If $2\lambda$ is not an integer, $u(r,\theta)=p(r)+\sum_{k=0}^{\infty}C_ke^{\frac{r^2}{4}}r^{-(2\lambda+m)}\big(1+g_k(\frac{r^2}{4})\big)\varphi_k$  with $p(r)=ce^{\frac{r^2}{4}}r^{-(m+2\lambda)}\int e^{-\frac{r^2}{4}}r^{m+4\lambda+1}(1+O(\frac{1}{r^2}))$ and $p(r)\sim r^{2\lambda}$;\\
If $\lambda$ is an integer, $u(r,\theta)=q(r,\theta)+C_0e^{\frac{r^2}{4}}r^{-(2\lambda+m)}\big(1+O(\frac{1}{r^2})\big)+\sum_{k\notin A}C_ke^{\frac{r^2}{4}}r^{-(2\lambda+m)}\big(1+g_k(\frac{r^2}{4})\big)\varphi_k$  with a polynomial $q(r,\theta)$ of degree $2\lambda$;\\
If $2\lambda$ is an integer but  $\lambda$ is not an integer, $u(r,\theta)=t(r,\theta)+\sum_{k\notin B}C_k e^{\frac{r^2}{4}}r^{-(2\lambda+m)}\big(1+g_k(\frac{r^2}{4})\big)\varphi_k$   with a polynomial $t(r,\theta)$ of degree $2\lambda.$ 
 \end{lemma}
\begin{proof}
Considering that $l_k = \frac{-(m-2) + \sqrt{(m-2)^2 + 4\lambda_k}}{4},$ we have
\begin{equation}
l_k-\lambda=-i \Longleftrightarrow k=-2(\lambda-i) \quad \text{or}  \quad k=m-2+2(\lambda-i).
\end{equation}
So we divide $\lambda$ into three cases.

Case 1. $2\lambda$ is not an integer.

By (3.7), we have $\frac{1}{\Gamma(l_k-\lambda)}\neq 0$ for any $k\ge 0.$ Then we can use Lemma 3.1 to get the asymptotic behavior of $f_k$ at infinity.
The first solution of (3.3) is
\begin{align*}
f_k(r)&=c_ke^{\frac{r^2}{4}}r^{2l_k}F[ l_k + \frac{m}{2}+\lambda,2l_k + \frac{m}{2};-\frac{r^2}{4}]\\
&=c_k\frac{\Gamma(2l_k+\frac{m}{2})}{\Gamma(l_k-\lambda)}e^{\frac{r^2}{4}}r^{-(2\lambda+m)}\big(1+g_k(\frac{r^2}{4})\big)\\
&=C_ke^{\frac{r^2}{4}}r^{-(2\lambda+m)}\big(1+g_k(\frac{r^2}{4})\big).
\end{align*}
The second solution of $f_0$ can be found by using the method of reduction of order. Let $(f_0)_2(r)=v(r)(f_0)_1(r)$ be the second linearly independent solution of (3.3), then
\[(f_0)_1(r)v''(r)+(2(f_0)'_1(r)+(\frac{m-1}{r} - \frac{r}{2})(f_0)_1)v'(r)=0.\]
Thus 
\begin{equation}
(f_0)_2(r)=c_0'(f_0)_1(r)\int\frac{e^{\frac{r^2}{4}}r^{1-m}}{((f_0)_1(r))^2}=ce^{\frac{r^2}{4}}r^{-(m+2\lambda)}\int e^{-\frac{r^2}{4}}r^{m+4\lambda+1}(1+O(\frac{1}{r^2}))=p(r)
\end{equation}
and  $p(r)\sim r^{2\lambda}.$
Thus 
\[u(r,\theta)=\sum_{k=0}^{\infty}f_k\varphi_k=p(r)+\sum_{k=0}^{\infty}C_ke^{\frac{r^2}{4}}r^{-(2\lambda+m)}\big(1+g_k(\frac{r^2}{4})\big)\varphi_k.\]

 Case 2. $\lambda$ is an integer.

$l_k-\lambda=-i $ for some integer $0\le i\le \lambda \Longleftrightarrow k\in A,$ where $A=\{0,m-2,m,m+2,\cdots,m+2\lambda\}.$
 
When $k=0,$ which implies that $i=\lambda,$ we have 
\[(f_0)_1(r)=c_0q_{2\lambda}(r).\] 
The second linearly independent solution can be found by using the method of reduction of order.
\begin{equation}
(f_0)_2(r)=c_0'(f_0)_1(r)\int\frac{e^{\frac{r^2}{4}}r^{1-m}}{((f_0)_1(r))^2}=l(r),
\end{equation}
where $l(r)=C_0e^{\frac{r^2}{4}}r^{-(2\lambda+m)}\big(1+O(\frac{1}{r^2})\big).$
When $k\ge 1,$ $f_k(r)=(f_k)_1(r),$ where $(f_k)_1(r)$ is the first solution of $f_k.$
Thus we have 
\begin{align*}
u(r,\theta)&=\sum_{k=0}^{\infty}f_k\varphi_k\\
&=\sum_{k\in A}f_k\varphi_k+\sum_{k\notin A}f_k\varphi_k\\
&=\sum_{k\in A}(f_k)_1\varphi_k+(f_0)_2+\sum_{k\notin A}(f_k)_1\varphi_k\\
&=\sum_{k\in A}c_ke^{\frac{r^2}{4}}r^{2l_k}F[ l_k + \frac{m}{2}+\lambda,2l_k + \frac{m}{2};-\frac{r^2}{4}]\varphi_k+l(r)\\
&+\sum_{k\notin A}c_ke^{\frac{r^2}{4}}r^{2l_k}F[ l_k + \frac{m}{2}+\lambda,2l_k + \frac{m}{2};-\frac{r^2}{4}]\varphi_k\\
&=\sum_{k\in A}c_{k}q_{k}   \varphi_{k}+l(r)+\sum_{k\notin A}C_ke^{\frac{r^2}{4}}r^{-(2\lambda+m)}\big(1+g_k(\frac{r^2}{4})\big)
\varphi_k\\
&=q(r,\theta)+l(r)+\sum_{k\notin A}C_ke^{\frac{r^2}{4}}r^{-(2\lambda+m)}\big(1+g_k(\frac{r^2}{4})\big)
\varphi_k,
\end{align*}
where $q(r,\theta)=\sum_{k\in A}c_{k}q_{k}\varphi_{k}$ is a polynomial of degree $2\lambda.$

Case 3. $2\lambda$ is an integer but  $\lambda$ is not an integer.

$l_k-\lambda=-i$ for some integer $0\le i\le [\lambda] \Longleftrightarrow k\in B,$ where $B=\{m-2,m,m+2,\cdots,m+2[\lambda]\}.$
 
Since $0\notin B, \frac{1}{\Gamma(l_k-\lambda)}\neq 0.$ By Lemma 3.1,  we have 
\begin{align*}
(f_0)_1(r)&=c_0e^{\frac{r^2}{4}}r^{2l_k}F[ l_k + \frac{m}{2}+\lambda,2l_k + \frac{m}{2};-\frac{r^2}{4}]\\
&=C_0e^{\frac{r^2}{4}}r^{-(2\lambda+m)}\big(1+g_0(\frac{r^2}{4})\big).
\end{align*}
The second linearly independent solution can be found by using the method of reduction of order.
\[(f_0)_2(r)=p(r),\]
where $p(r)\sim r^{2\lambda}.$\\
When $k\ge 1,$ $f_k(r)=(f_k)_1(r),$ where $(f_k)_1(r)$ is the first solution of $f_k.$
Thus we have 
\begin{align*}
u(r,\theta)&=\sum_{k=0}^{\infty}f_k\varphi_k\\
&=\sum_{k\in B}f_k\varphi_k+\sum_{k\notin B}f_k\varphi_k\\
&=\sum_{k\in B}(f_k)_1\varphi_k+(f_0)_2+\sum_{k\notin B}(f_k)_1\varphi_k\\
&=\sum_{k\in B}c_ke^{\frac{r^2}{4}}r^{2l_k}F[ l_k + \frac{m}{2}+\lambda,2l_k + \frac{m}{2};-\frac{r^2}{4}]\varphi_k+p(r)\\
&+\sum_{k\notin B}c_ke^{\frac{r^2}{4}}r^{2l_k}F[ l_k + \frac{m}{2}+\lambda,2l_k + \frac{m}{2};-\frac{r^2}{4}]\varphi_k\\
&=\sum_{k\in B}c_{k}q_{k}   \varphi_{k}+p(r)+\sum_{k\notin B}C_ke^{\frac{r^2}{4}}r^{2\lambda+m}\big(1+g_k(\frac{r^2}{4})\big)\varphi_k\\
&=\sum_{k\in B}c_{k}q_{k}   \varphi_{k}+p(r)+\sum_{k\notin B}C_ke^{\frac{r^2}{4}}r^{2\lambda+m}\big(1+g_k(\frac{r^2}{4})\big)
\varphi_k\\
&=s(r,\theta)+p(r)+\sum_{k\notin B}C_ke^{\frac{r^2}{4}}r^{2\lambda+m}\big(1+g_k(\frac{r^2}{4})\big)
\varphi_k,
\end{align*}
where $s(r,\theta)=\sum_{k\in B}c_{k}q_{k}   \varphi_{k}$ is a polynomial of degree $2\lambda.$
If we denote $t(r,\theta)=s(r,\theta)+p(r),$ by Lemma 1.2 in \cite{5}, we know that $ t(r,\theta)$ is an polynomial of degree $2\lambda.$  Then 
\[u(r,\theta)=t(r,\theta)+\sum_{k\notin B}C_ke^{\frac{r^2}{4}}r^{2\lambda+m}\big(1+g_k(\frac{r^2}{4})\big)\varphi_k.\]  
\end{proof}

 \begin{proof}[Proof of Theorem \ref{thm:1.3}]
We divide $\lambda$ into three cases:

Case1. $2\lambda$ is not an integer.

For any fixed $k\ge 0,$ note that 
\begin{align*}
\frac{1}{2}|C_k|e^{\frac{r_i^2}{4}}r_i^{-(m+2\lambda)}& \le |C_k||1+g_k(\frac{r_i^2}{4})|e^{\frac{r_i^2}{4}}r_i^{-(m+2\lambda)}\\
&=|\int_{S^{m-1}}u(r_i,\theta)\varphi_k|\\
&\le (\int_{S^{m-1}}( u(r_i,\theta))^2)^{\frac{1}{2}}(\int_{S^{m-1}} \varphi_k^2)^{\frac{1}{2}}\\
&\le C\sqrt{\omega_n}e^{\frac{r_i^2}{4}}r_i^{-(m+2\lambda+\epsilon)},
\end{align*}
 which implies that 
\[C_k<Cr_i^{-\epsilon}.\]
Let $r_i\rightarrow+\infty,$  then
\[C_k\equiv 0.\]
It is clear that 
\[u(r,\theta)\equiv p(r).\]

Case 2. $\lambda$ is an integer.

Similarly, we can obtain that
\[u(r,\theta)\equiv q(r,\theta).\]

Case 3. $2\lambda$ is an integer but  $\lambda$ is not an integer.

Similarly, we can obtain that
\[u(r,\theta)\equiv t(r,\theta).\]
\end{proof}

 \begin{proof}[Proof of Theorem \ref{thm:1.7}]
We divide $\lambda$ into three cases:

Case 1. $2\lambda$ is not an integer.

For any $g(\theta)\in H^{[\frac{m}{2}]+2}(S^{m-1}),$ by Lemma 2.4, we have $g(\theta)=\sum_{k=0}^{\infty} \bar C_k\varphi_k$ and $|\bar C_k|\le \frac{C}{\lambda_k^{\frac{1}{2}[\frac{m}{2}]+1}}.$ By Lemma 3.2, for any $k\ge 0,$
\[\bar f_k= \bar C_ke^{\frac{r^2}{4}}r^{-(2\lambda+m)}\Big(1+g_k(\frac{r^2}{4})\Big) \]
is a solution of (3.3). 
Assume  $L=2\delta,$ by using the fact that $|\varphi_k|^2\le C(m)k^{m-2},$ we have
\[|\bar f_k \varphi_k|\le Ce^{\frac{r^2}{4}}r^{-(m+2\lambda)}\frac{1}{k^{2}}\Big(1+C(m,\delta,\lambda)\frac{k^{2\delta}}{r^{2\delta}}\Big)\] 
for $k$ sufficient large, where $C$ is independent of $k$. Assume $L=1+\delta,$ similarly we have
\[|\bar f_k \varphi_k|\le Ce^{\frac{r^2}{4}}r^{-(m+2\lambda)}\frac{1}{k^{2}}\Big(1+C(m,\delta,K,\lambda)\frac{1}{r^2}\Big)\]
when $k\le K$ for some $K>0$, where $C$ is independent of $k$.
If we set $0<\delta<\frac{1}{2},$ the series $\sum_{k=0}^{\infty} \bar f_k\varphi_k$ is uniformly convergent on $\overline { B_A(0)}\setminus B_{\epsilon}(0)$ for any $\epsilon,A>0.$ Assume $\bar u=\sum_{k=0}^{\infty} \bar f_k\varphi_k,$ then we have 
\begin{align*}
\Delta_g(\bar u)&=\bar u_{r r}+\frac{m-1}{r}\bar u_r+\frac{1}{r^2}\Delta_{\theta}\bar u-\frac{r}{2}\frac{\partial \bar u}{\partial r}\\
&=\sum_{k=1}^{\infty} \Big((\bar f_k)_{rr}+(\frac{m-1}{r}-\frac{r}{2})(\bar f_k)_r-\frac{\lambda_k\bar f_k}{r^2}\Big)\varphi_k+(\bar f_0)_{r r}+(\frac{m-1}{r}-\frac{r}{2})(\bar f_0)_r\\
&=-\lambda\sum_{k=0}^{\infty}\bar f_k\varphi_k\\
&=-\lambda \bar u
\end{align*}
Thus $\bar u$ is a solution of $\Delta_g u=-\lambda u$ on $\mathbb{R}^m\backslash \{0\}.$\\
Since $\sum_{k=0}^{\infty} \frac{\bar f_k}{u_0}\varphi_k$ is uniformly convergent on $\mathbb{R}^m\setminus B_{\epsilon}(0)$ for any $\epsilon>0$, then we can obtain that 
\begin{align*}
\lim\limits_{r\to+\infty}\frac{\bar u(r,\theta)}{u_0(r)}=\lim\limits_{r\to+\infty}\sum_{k=0}^{\infty} \frac{\bar f_k}{u_0}\varphi_k=\sum_{k=0}^{\infty} \lim\limits_{r\to + \infty}\frac{\bar f_k}{u_0}\varphi_k=g(\theta).
\end{align*}
Hence $\bar u$ is the solution of $\Delta_h u=-\lambda u$ which satisfies that $\lim\limits_{r\to+\infty}\frac{\bar u(r,\theta)}{u_0(r)}=g(\theta)$ for any given function $g(\theta)\in H^{[\frac{m}{2}]+2}(S^{m-1}).$
Moreover, assume $\frac{\bar u(r,\theta)}{u_0(r)}$ and $\frac{ u(r,\theta)}{u_0(r)}$ are asymptotic to the same function $g(\theta)$, then 
\begin{align*}
&C_k=\lim\limits_{r\to+\infty} \frac{ f_k}{u_0}=\langle\lim\limits_{r\to+\infty} \frac{ u(r,\theta)}{u_0(r)}, \varphi_k \rangle \\
&=\langle \lim\limits_{r\to+\infty}\frac{\bar u(r,\theta)}{u_0(r)}, \varphi_k \rangle=\lim\limits_{r\to + \infty}\frac{\bar f_k}{u_0}=\bar C_k.
\end{align*}
Hence 
\begin{align*}
&u(r,\theta)-\bar u(r, \theta)\\
&=p(r)-\bar p(r)+\sum_{k=0}^{\infty} (C_k-\bar C_k)e^{\frac{r^2}{4}}r^{-(2\lambda+m)}\Big(1+g_k(\frac{r^2}{4})\Big)\varphi_k\\
&=p(r)-\bar p(r).
\end{align*}
 Case 2. $\lambda$ is an integer.

For any $g(\theta)\in H^{[\frac{m}{2}]+2}(S^{m-1})$ satisfying that $\langle g,\varphi_k\rangle=0$ when $k\in \{0,m-2,m,m+2,\cdots,m+2\lambda\}=A,$  by Lemma 2.4, we have $g(\theta)=\sum_{k\notin A}\bar C_k\varphi_k$ and $|\bar C_k|\le \frac{C}{\lambda_k^{\frac{1}{2}[\frac{m}{2}]+1}}.$\\ 
Set \[u(r,\theta)=\sum_{k\notin A}\bar C_ke^{\frac{r^2}{4}}r^{-(2\lambda+m)}\Big(1+g_k(\frac{r^2}{4})\Big)\varphi_k+\bar C_0l(r).\]
Then $\bar u$ is a solution of $\Delta_h=-\lambda u$ which satisfies that $\lim\limits_{r\to+\infty}\frac{\bar u(r,\theta)}{u_0(r)}=g(\theta).$

Case 3. $2\lambda$ is an integer but  $\lambda$ is not an integer.

For any $g(\theta)\in H^{[\frac{m}{2}]+2}(S^{m-1})$ satisfying that $\langle g,\varphi_k\rangle=0$ when $k\in \{m-2,m,m+2,\cdots,m+2[\lambda]\}=B,$ by Lemma 2.4, we have $g(\theta)=\sum_{k\notin B}\bar C_k\varphi_k$ and $|\bar C_k|\le \frac{C}{\lambda_k^{\frac{1}{2}[\frac{m}{2}]+1}}.$\\ 
Set \[u(r,\theta)=\sum_{k\notin B}\bar C_ke^{\frac{r^2}{4}}r^{-(2\lambda+m)}\Big(1+g_k(\frac{r^2}{4})\Big)\varphi_k.\]
Then $\bar u$ is a solution of $\Delta_h u =-\lambda u$ which satisfies that $\lim\limits_{r\to+\infty}\frac{\bar u(r,\theta)}{u_0(r)}=g(\theta).$
\end{proof}


\end{document}